\documentclass[11pt]{article}
\usepackage{mathrsfs}
\usepackage{amsmath, amssymb, amsthm, latexsym, amsfonts, epsfig, color,graphicx}

\usepackage{hyperref} 

\begin{document}

\newcommand{\blue}{\color{blue}}
\newcommand{\red}{\color{red}}
\newcommand{\magenta}{\color{magenta}}

\newcommand{\s}{\sigma}
\renewcommand{\k}{\kappa}
\newcommand{\p}{\partial}
\newcommand{\D}{\Delta}
\newcommand{\om}{\omega}
\newcommand{\Om}{\Omega}
\renewcommand{\phi}{\varphi}
\newcommand{\e}{\epsilon}
\renewcommand{\a}{\alpha}
\renewcommand{\b}{\beta}
\newcommand{\N}{{\mathbb N}}
\newcommand{\R}{{\mathbb R}}
   \newcommand{\eps}{\varepsilon}
   \newcommand{\EX}{{\mathbb{E}}}
   \newcommand{\PX}{{\mathbb{P}}}

\newcommand{\cF}{{\cal F}}
\newcommand{\cG}{{\cal G}}
\newcommand{\cD}{{\cal D}}
\newcommand{\cO}{{\cal O}}

\newcommand{\de}{\delta}

\newcommand{\dif}{{\mathord{{\rm d}}}}

\newcommand{\grad}{\nabla}
\newcommand{\n}{\nabla}
\newcommand{\curl}{\nabla \times}
\newcommand{\dive}{\nabla \cdot}

\newcommand{\ddt}{\frac{d}{dt}} 
\newcommand{\la}{{\lambda}}

\newtheorem{theorem}{Theorem}[section]
\newtheorem{lemma}{Lemma}[section]
\newtheorem{remark}{Remark}[section]
\newtheorem{example}{Example}[section]
\newtheorem{definition}{Definition}[section]
\newtheorem{corollary}{Corollary}[section]
\newtheorem{assumption}{Assumption}[section]
\newtheorem{prop}{Proposition}[section]
\newtheorem{notation}{Notation}[section]
\def\proof{\mbox {\it Proof.~}}
\makeatletter
\@addtoreset{equation}{section}
\makeatother
\renewcommand{\theequation}{\arabic{section}.\arabic{equation}}

 \makeatletter\def\theequation{\arabic{section}.\arabic{equation}}\makeatother

\title{ Synchronization by degenerate noise }

\author{
{ \bf\large Xianming Liu, Xu Sun}\hspace{2mm}
\vspace{1mm}\vspace{1mm}\\
{\it\small  School of Mathematics and Statistics},\\
{\it\small Huazhong University of Science and Technology},
{\it\small Wuhan 430074,  China}\\
{\it\small E-mail:    xmliu@hust.edu.cn, xsun@hust.edu.cn }\vspace{1mm}}

\date{\today } 
\maketitle

\begin{abstract} 
In this paper, we derive several criteria for (weak) synchronization by noise without the global swift transitivity property. Our sufficient conditions for (weak) synchronization are necessary and can be applied to scenarios involving degenerate or non-Gaussian noise.
These results partially answer the open question posed by Flandoli et al. (Probab Theory Relat Fields 168:511-556, 2017).
As an application, we prove that the weak attractor for stochastic Lorenz 63 systems driven by degenerate noise consists of a single random point provided the noise intensity is small, and there is no weak synchronization if the noise intensity is large. This indicates that a bifurcation occurs in relation to the intensity of the noise.

\end{abstract}


\setcounter{secnumdepth}{5} \setcounter{tocdepth}{5}

\makeatletter
    \newcommand\figcaption{\def\@captype{figure}\caption}
    \newcommand\tabcaption{\def\@captype{table}\caption}
\makeatother




\section{\bf Introduction }  \label{Intro}

In this paper, we introduce sufficient conditions for (weak) synchronization by general white noise, which could be degenerate or non-Gaussian. Here, (weak) synchronization by noise means that there is a random (point) attractor $A$ for random dynamical systems (RDS) $\phi$ consisting of a single random point, i.e., $A(\omega)=a(\omega)$ almost surely. 

It is worth pointing out that if the state space $E$ is local and $\sigma$-compact,  our sufficient conditions for synchronization by noise are necessary simultaneously. 
On the other hand, it is well known that weak synchronization implies that, for each $x,y\in E$,
\begin{equation}\label{eq:1.1}
\lim_{t \to \infty}  d(\phi_t(\omega, x), \phi_t(\omega, y)) = 0 \quad \text{ in probability}.
\end{equation}
Assume that a white noise RDS $\phi$ has right-continuous trajectories and is strongly mixing. We obtain that
$$
\text{Weak synchronization} \Leftrightarrow  
\eqref{eq:1.1}.
$$
Therefore, these theorems can be applied to scenarios involving degenerate noise or non-Gaussian noise. 

As an application, we study stochastic Lorenz 63 systems driven by degenerate noise~\cite{CH21}. The associated RDS synchronizes if the noise intensity is small, and there is no weak synchronization provided the noise intensity is large.  

\subsection{\bf Literature review }  \label{Literature}

The terminology associated with synchronization by noise is inconsistently utilized in the literature. This has resulted in the categorization of at least three distinct phenomena under this overarching term.

(1) \textbf{Stabilization by noise.} In some instances, the effect that deterministic invariant points may become stable due to the inclusion of noise has been referred to as stabilization by noise~\cite{CR04}. In some examples, the deterministic and stochastic systems share the same deterministic invariant points. The property that every two trajectories of a noisy system
converge to one another, that is, \eqref{eq:1.1} holds, has been named synchronization by
noise in recent publications (e.g., \cite{Newman}). This property is closely related and a simple consequence of the results for weak synchronization obtained in~\cite{FGS17b} and Corollary~\ref{cor:2.2} in this work.  {\em``Thus, it would be a good reason to refer to the effects observed here as stabilization by noise"}~\cite{FGS17b}. 

(2) \textbf{Coupled synchronization.} The synchronization of coupled dynamical systems is a well-known phenomenon. 
It deals with coupled dynamical systems with a common dynamical feature in an asymptotic sense~\cite{Caraballo06, CK05}. It seems that coupled synchronization predominantly relies on the methods of coupling employed, even in the absence of noise.

(3) \textbf{Synchronization by noise.} In some cases, noise destroys unstable deterministic invariant points~\cite{CF98}.  As a result, a deterministic set attractor transforms into a random attractor that consists of a single random point by noise, see~\cite{Tearne, Caraballo07, FGS17}. Synchronization by noise for order-preserving RDS has been analyzed~\cite{FGS17b}. Synchronization by noise for SPDEs has been investigated~\cite{GT24}.

In this paper, we mainly study the third case, i.e., synchronization by noise.

\subsection{\bf Motivations and comments }  \label{comments}
The motivation for this work comes from \cite{FGS17}. 
In \cite{FGS17}, the author showed that locally asymptotically stable, contracting on large sets, and globally swift transitive imply synchronization~\cite[Theorem 2.14]{FGS17}.  Weak synchronization occurs provided a strongly mixing white noise RDS is weakly asymptotically stable, pointwise strongly swift transitive ~\cite[ Theorem 2.23]{FGS17}.  
In the same article, the authors said
{ \em``
Numerical evidence suggests that the top Lyapunov exponent for the Lorenz system perturbed by strong noise 
is negative and (weak) synchronization occurs. The Lorenz system, however, is not covered by the techniques put forward in Sect. 3. We leave this as an open problem. We prove swift transitivity for a large class of SDE with (non-degenerate) additive noise in Sect. 3. It is left as an open problem to establish swift transitivity in other situations, such as degenerate additive or multiplicative noise."}

Now, the main challenges and contributions are summarized below.

(1) It is well known that weak synchronization implies that, for each $x,y\in E$,
\begin{equation}\label{eq:1.1-aga}
\lim_{t \to \infty}  d(\phi_t(\omega, x), \phi_t(\omega, y)) = 0
\end{equation}
in probability. So weak asymptotic stability and the pointwise stability condition are also necessary for weak synchronization, while global pointwise strong swift transitivity is not.
Assume that $\phi$ has right-continuous trajectories and is strongly mixing. We obtain 
$$
\text{Weak synchronization} \Leftrightarrow \lim_{t \to \infty}  d(\phi_t(\omega, x), \phi_t(\omega, y)) = 0 \ \text{ in probability}, \forall x,y\in E.
$$
Even local pointwise strong swift transitivity is not needed, see Theorem~\ref{thm:2.3}.  

(2) Concerning synchronization by noise, we replace contraction on large sets (necessary condition) and global swift transitivity (unnecessary condition) with local strong swift transitivity (necessary condition).
If the state space $E$ is Heine-Borel,  our sufficient conditions for synchronization by noise, with no need to verify contraction on large sets, are necessary simultaneously, see Theorem~\ref{thm:2.1}.  As a result, these theorems can be applied to scenarios involving degenerate noise or non-Gaussian noise.

(3) In \cite{FGS17}, the authors conjecture that the top Lyapunov exponent for the Lorenz system perturbed by strong noise is negative and (weak) synchronization occurs. Also inspired by the work of M. Coti~Zelati and M. Hairer~\cite{CH21}, we study stochastic Lorenz 63 systems driven by degenerate noise. The associated RDS does synchronize if the noise intensity is small, and there is no weak synchronization provided that the noise intensity is large. This means a bifurcation appears with respect to the noise intensity.

\subsection{\bf Organization of the article}  \label{Org}

In Section~\ref{pre}, we introduce several definitions and present new criteria for synchronization and weak synchronization. In Section~\ref{Syn SDE}, we establish results regarding local asymptotic stability and local strong swift transitivity for several classes of SDEs. In Section~\ref{Sec:example}, we investigate the synchronization for three specific types of degenerate SDEs, including the stochastic Lorenz 63 system. In Section~\ref{proof}, we provide the proofs of the theorems presented in Section~\ref{pre}. 


\section{\bf Preliminaries and main results }  \label{pre}


Let $(E,d)$ be a Polish space. We put the notation and background on random dynamical systems (RDS) in the appendix. We refer the reader to~\cite{Arn98} for further reading.


\subsection{\bf Synchronization } \label{S21}

\begin{definition}
    We say that synchronization occurs for a given {\em RDS} $\phi$ if there is a weak attractor $A(\omega)$ being a singleton, for $\PX$-a.e. $\omega \in \Omega$.
\end{definition}

We first recall some definitions and a sufficient condition for synchronization~\cite{FGS17}.
\begin{definition}\label{def:AS} {\em(\cite[Definition 2.2]{FGS17})}
Let $U \subset E$ be a non-empty open set.  We say $\phi$ is asymptotically stable on $U$ if there exists a  time sequence  $t_n  \uparrow \infty$ such that  
$$
\PX \left( \lim_{n\rightarrow \infty }{\rm{diam}}(\phi_{t_n}(\cdot, U)) =0\right)>0,
$$
where $\mathrm{diam}(A):=\sup\limits_{x,y\in A} d(x,y)$.
\end{definition}

\begin{definition} \label{def:CLS} {\em(\cite[Definition 2.11]{FGS17})}
We say $\phi$ is contracting on large sets  if for every $R>0$, there is  a ball $B(y,R)$ a time $t >0$ such that  
$$
\PX \left( {\rm{diam}} \left(\phi_t(\cdot, B(y,R))\right) \leq \frac{R}{4} \right)>0.
$$
\end{definition}

\begin{definition} \label{def:ST} {\em(\cite[Definition 2.7]{FGS17})}
We say $\phi$ is swift transitive if, 
for every $x,y\in E$ and $R>0$, there is a time $t_0>0$ such that  
$$
\PX \left( \phi_{t_0}(\cdot,B(x,R)) \subset B (y, 2 R) \right)>0.
$$
\end{definition}

\begin{prop}\label{prop:2.1} {\em(\cite[Theorem 2.14]{FGS17})}
 Let $A$ be the unique weak attractor of an {\em RDS} $\phi$.  $\phi$ is locally asymptotically stable, contracting on large sets and globally swift transitive. Then synchronization occurs.
\end{prop}

\begin{remark} 
Swift transitivity is an irreducibility condition. Global swift transitivity often fails for SDEs with L\'evy noise or degenerate Gaussian noise. In general, Proposition~\ref{prop:2.1} could not be applied to scenarios involving degenerate noise or non-Gaussian noise. For specific examples, please refer to Remark~\ref{rem:2.4}.
\end{remark}

We give some definitions that can be seen as local versions of swift transitivity.
\begin{definition}\label{def:loc-ST} 
We say $\phi$ is weakly swift transitive at $z$ if, 
for every $R>0,x\in E$, there is a time $t_0>0$ such that  
$$
\PX \left( \phi_{t_0}(\cdot,B(x,R)) \subset B (z, 2 R) \right)>0.
$$
We say $\phi$ is weakly recurrent at $z$ if, 
for every $R>0$, there is a time $t_0>0$ such that  
$$
\PX \left( \phi_{t_0}(\cdot,B(z,R)) \subset B (z, 2 R) \right)>0.
$$
We say $\phi$ is strongly swift transitive at $z$ if, 
for every $R>0,x\in E$, there is a time $t_0 >0$ such that  
$$
\PX \left( \phi_{t_0}(\cdot,B(x,R)) \subset B \bigg(z, \frac{R}{2}\bigg) \right)>0.
$$
We say $\phi$ is strongly recurrent at $z$ if, 
for every $R>0$, there is a time $t_0 >0$ such that  
$$
\PX \left( \phi_{t_0}(\cdot,B(z,R)) \subset B \bigg(z, \frac{R}{2}\bigg) \right)>0.
$$
\end{definition}

We have the following result, which gives a necessary and sufficient criterion for synchronization.
\begin{theorem}  \label{thm:2.1}
Let $E$ be locally compact and $\sigma$-compact and $A$ be the weak attractor of an {\em{RDS}} $\phi$. Then the following statements are equivalent:\\ 
{\rm(i)} Synchronization occurs, i.e. $A(\omega)=\{a(\omega)\}$.\\
{\rm(ii)} There exists a point $z\in E$  and  $\varepsilon>0$ such that  $\phi$ is asymptotically stable on $B(z,\varepsilon)$ and  $\phi$ is strongly swift transitive at $z$.\\
{\rm(iii)} There exists a point $z\in E$  and  $\varepsilon>0$ such that  $\phi$ is asymptotically stable on $B(z,\varepsilon)$ and  $\phi$ is strongly recurrent at $z$.\\
{\rm(iv)} There exists a point $z\in E$  and  $\varepsilon>0$ such that  $\phi$ is asymptotically stable on $B(z,\varepsilon)$ and  $\PX(A\subset B(z,\varepsilon))>0.$ 
\end{theorem} 
\begin{remark} 
Every metric space $(E,d)$ that is locally compact and $\sigma$-compact allows an equivalent metric $d'$ such that $(E, d')$ is Heine-Borel.  As a result, a weak attractor attracts all closed and bounded subsets (\cite{FGS17}). Every finite-dimensional Euclidean space is locally compact and $\sigma$-compact. 
\end{remark}
\begin{remark} 
Assume that $E$ is locally compact, $\sigma$-compact and synchronization holds, then there exists a point $z\in E$  and  $\varepsilon>0$ such that  $\phi$ is asymptotically stable on $B(z,\varepsilon)$,   contracting on large sets and weakly swift transitive at $z$. For proof, we refer to Section 5.   Conversely, it seems that the above conditions (locally asymptotically stable, contracting on large sets, locally swift transitive) also imply synchronization, but we cannot guarantee this at the moment.
\end{remark}
\begin{remark}\label{rem:2.4}   Consider the one-dimensional SDE \cite[Example 2.15]{FGS17}
   $$
   \dif X_t= X_t \dif W_t. 
   $$
 $A(\omega)=\{0\}$ is the weak attractor, synchronization occurs, and swift transitivity fails. So Proposition ~\ref{prop:2.1} can not be applied, but Theorem~\ref{thm:2.1} can be.  
 
Another example from  \cite[Example 3.14]{HM06} reads
$$
   \dif X_t= -X_t \dif t, \quad \dif Y_t= (Y_t-Y_t^3)\dif t +\dif W_t.
$$
Which is an SDE system with highly degenerate noise that does not even satisfy the H\"ormander condition.
Due to Theorem~\ref{thm:2.1}, local swift transitivity property and synchronization hold.  But the global swift transitivity fails again.  
\end{remark}

Without the additional compact assumptions on $E$, we have the following corollary on sufficient conditions for synchronization.
\begin{corollary}  \label{cor:2.1}
 Let $A$ be the unique weak attractor of an {\em RDS} $\phi$. If one of the following statements holds, then synchronization occurs.\\
{\rm(i)} There exists a point $z\in E$  and  $\varepsilon>0$ such that  $\phi$ is asymptotically stable on $B(z,\varepsilon)$ and  $\PX(A\subset B(z,\varepsilon))>0.$ \\ 
{\rm(ii)} There exists a point $z\in E$  and  $\varepsilon>0$ such that  $\phi$ is asymptotically stable on $B(z,\varepsilon)$ and  $\phi$ is strongly swift transitive at $z$.\\
{\rm(iii)} There exists a point $z\in E$  and  $\varepsilon>0$ such that  $\phi$ is asymptotically stable on $B(z,\varepsilon)$ and  $\phi$ is strongly recurrent at $z$.
\end{corollary}
\begin{remark}     
The proof of Corollary~\ref{cor:2.1} lies in the proof of Theorem~\ref{thm:2.1}. 
This result may be useful for SPDEs driven by degenerate additive noise or multiplicative noise.
\end{remark}


\subsection{\bf Weak synchronization } \label{S22}

We will assume throughout this subsection that $\phi$ is a white noise RDS and $\rho$ is the Markovian invariant measure.  

\begin{definition}
    We say that weak synchronization occurs for a given {\em RDS} $\phi$ if there is a minimal weak point attractor $A(\omega)$ being a singleton, for $\PX$-a.e. $\omega \in \Omega$.
\end{definition}

It is easy to see that synchronization implies weak synchronization. We recall some definitions and a sufficient condition for weak synchronization~\cite{FGS17}.

\begin{definition}\label{def:WAS} {\em(\cite[Definition 2.17]{FGS17})}
Let $U \subset E$ be a (deterministic) non-empty open set. We say that $\phi$ is weakly asymptotically stable on $U$ if there exists a (deterministic) sequence $t_n \uparrow \infty$, and a set $\mathcal{M} \subset \Omega$ of positive $\PX$-measure, such that for all $x, y \in U$,
\[
\lim_{n \to \infty}  1_{\mathcal{M}}   d(\phi_{t_n}(\cdot, x), \phi_{t_n}(\cdot, y))  = 0, \ \text{in probability.}
\]
\end{definition}

\begin{definition}{\em(\cite[Definition 2.22]{FGS17})}
We say that $\phi$ is pointwise strongly swift transitive if there is a time $t > 0$ such that for every $x, x' \in E$ and every (arrival) point $y$,
\[
P\left( \phi_t(\cdot,\{x,x'\}) \subset B(y,2d(x,x'))\right) > 0.
\]
\end{definition}

\begin{prop}\label{prop:2.2} {\em(\cite[Theorem 2.23]{FGS17})}
Assume that $\phi$ has right-continuous trajectories, is strongly mixing, weakly asymptotically stable on $U$ with $\rho(U) > 0$, pointwise strongly swift transitive, and
\begin{equation}\label{eq:2.1}
\liminf_{t \to \infty}  d(\phi_t(\omega, x), \phi_t(\omega, y)) = 0 \quad \text{a.s.},\quad \forall x, y \in E
\end{equation} 
Then, there is a minimal weak point attractor $A$ consisting of a single random point $a(\omega)$ and $A(\omega)  = \{a(\omega)\} \quad \PX\text{-a.s.}$, i.e. weak synchronization holds.
\end{prop}

\begin{remark}
 The pointwise strongly swift transitive property is too strong to verify.  For a simple example,
$
\dot{x}=-x.
$
The pointwise strongly swift transitive property fails. Proposition~\ref{prop:2.2} can not be applied, but (weak) synchronization holds. 
\end{remark}

As an extension of the globally pointwise strongly swift transitive property, 
we introduce the following definition. 

\begin{definition}
We say that $\phi$ is pointwise strongly swift transitive at $z$ if there is a time $t > 0$ such that for every $x\neq x' \in E$,
\[
P\left( \phi_t(\cdot,\{x,x'\}) \subset B(z,2d(x,x'))\right) > 0.
\]
\end{definition}

The following result is a slight extension of Proposition~\ref{prop:2.2}.
\begin{theorem}\label{thm:2.2}
Assume that $\phi$ has right-continuous trajectories and is strongly mixing.  There exist a point $z\in \mathrm{supp}(\rho)$  and  $\varepsilon>0$ such that $\phi$ is weakly asymptotically stable on $B(z, \varepsilon)$,  pointwise strongly swift transitive at $z$, and
\begin{equation}\label{eq:2.3}
\liminf_{t \to \infty}  d(\phi_t(\omega, x), \phi_t(\omega, y)) = 0 \quad \text{a.s.},\quad \forall x, y \in E
\end{equation} 
Then, there is a minimal weak point attractor $A$ consisting of a single random point $a(\omega)$ and $A(\omega) = \{a(\omega)\}~ \PX\text{-a.s.}$, i.e. weak synchronization holds.
\end{theorem}

\begin{remark}
It seems that the locally pointwise strongly swift transitive property is not yet a necessary condition for weak synchronization.
\end{remark}

If weak synchronization occurs, then for all $x,y\in E$, 
$$
\lim_{t \to \infty}  d(\phi_t(\omega, x), \phi_t(\omega, y)) = 0 \quad \text{ in probability}.
$$
This implies weak asymptotic stability and equation~\eqref{eq:2.1}. In contrast to Proposition~\ref{prop:2.2}, we have the following result.

\begin{theorem}\label{thm:2.3}
Assume that $\phi$ has right-continuous trajectories. The following statements are equivalent:\\ 
{\rm(i)} Weak synchronization occurs.\\
{\rm(ii)} $\phi$ is strongly mixing,   and for all  $x,y\in E$,
\begin{equation}\label{eq:mixing}
\lim\limits_{t \to \infty}  d(\phi_t(\omega, x), \phi_t(\omega, y)) = 0 \ \text{ in probability}.
\end{equation}
\end{theorem}

Recall that $\phi$ is monotone with respect to initial conditions (see \cite{Chueshov02}) if for each $t,\omega$ and $x\leq y$, 
$$
\phi(t,\omega,x)\leq \phi(t,\omega,y).
$$
\begin{corollary}\label{cor:2.2}
Assume that a monotone RDS $\phi$ has right-continuous trajectories. The following  statements are equivalent:\\ 
{\rm(i)}  $\phi$ is strongly mixing. \\
{\rm(ii)} Weak synchronization occurs.
\end{corollary}
\begin{proof}
Since $\phi$ is strongly mixing, then for any $x\in E$,  $\phi_t(\cdot, x)$ weakly converges to the invariant measure $\rho$ as $t \rightarrow \infty$.
By \cite[Proposition 2.4]{FGS17b}, we have for all  $x,y\in E$,
\begin{equation*} 
\lim\limits_{t \to \infty}  d(\phi_t(\omega, x), \phi_t(\omega, y)) = 0 \ \text{ in probability}.
\end{equation*}
Thanks to Theorem~\ref{thm:2.3}, weak synchronization holds. \qed
\end{proof}

\begin{remark}
In \cite[Theorem 2.6]{FGS17b}, the authors derive a sufficient condition for weak synchronization when $\phi$  may not be a white noise RDS. Corollary~\ref{cor:2.2} can be seen as a weak version of \cite[Theorem 2.6]{FGS17b}. 

A white noise $\phi$ is strongly mixing, then $\phi$ is completely mixing, i,e., 
for every $x,y\in E$, every bounded Lipschitz function $f$,
$$
 \lim_{t\rightarrow \infty} |P_tf(x)-P_tf(y)|=0.
$$
This, together with the monotonicity, implies  
\begin{equation*} 
\lim\limits_{t \to \infty}  d(\phi_t(\omega, x), \phi_t(\omega, y)) = 0 \ \text{ in probability}.
\end{equation*}
An interesting question arises regarding the additional conditions, aside from monotonicity, that can facilitate the conversion of weak convergence into strong convergence.
\end{remark}


\section{\bf Synchronization for SDEs driven by white noise}  \label{Syn SDE}



\subsection{\bf Local asymptotic stability}  \label{S31}


Recall a result from \cite[Lemma 3.1]{FGS17} on the Lyapunov exponent, which implies local asymptotic stability.
\begin{prop}\label{prop:3.1} Let $\phi_t(\omega,\cdot) \in C_{\mathrm{loc}}^{1,\delta} $ for some  $\delta \in (0,1) $  and all  $ t \geq 0 $ and let $ P_t$ be the Markovian semigroup associated to $\phi$. Assume that $P$ has an ergodic invariant measure $\rho$ such that
\begin{equation}\label{eq:IC}
\begin{aligned}
&\EX\int\log^+ \|D \phi_1(w,x)\| \, \dif \rho(x) < \infty\\
&\EX\int\log^+ \|\phi_1(\omega, \cdot + x) -\phi_ 1(\omega,x)\|_{\mathcal{C}^{1,\delta}(\bar{B}(0,1))} \, \dif \rho(x) < \infty. 
\end{aligned}    
\end{equation}
Then\\v
\rm{(i)} there are constants $\lambda_N < ... < \lambda_1$ such that
\[
\lim_{m \to \infty} \frac{1}{m} \log |D \phi_m(\omega,x) v| \in \{\lambda_i\}_{i=1}^N,
\]
 for all $v \in \mathbb{R}^d \setminus \{0\}$  and almost all $(\omega,x) \in \Omega \times \mathbb{R}^d.$\\
\rm{(ii)} Assume that the top Lyapunov exponent $\lambda_{\mathrm{top}} := \lambda_1< 0$. Then for every $\varepsilon\in (\lambda_{\mathrm{top}},0)$,  there is a measurable map $\beta: \Omega \times \mathbb{R}^d \to \mathbb{R}_+\setminus\{0\}$ such that for $\rho$-almost all $x \in \mathbb{R}^d$,
\[
\mathcal{S}(\omega,x) := \{ y \in \mathbb{R}^d : | \phi_m(\omega,y) - \phi_m(\omega,x) | \leq \beta(\omega,x) e^{\varepsilon m} \text{ for all } m \in \mathbb{N} \}
\]
is an open neighborhood of $x$, $\PX$-almost surely.
\end{prop}

A vector field $b: \mathbb{R}^d \to \mathbb{R}^d$ is said to be eventually strictly monotone if there exists an $R > 0$ such that
\begin{equation}\label{eq:drift}
(b(x) - b(y), x - y) \leq 
\begin{cases}
\eta |x - y|^2, &\quad \text{for all } |x|+|y| < R\\
-\lambda |x - y|^2, &\quad \text{for all } |x|+|y| \geq R
\end{cases}
\end{equation}
for some $\eta,\lambda > 0$.

The above condition is a slight extension of eventually strictly monotone in~\cite{FGS17}. We still call it eventually strictly monotone.  This condition, together with the Young inequality, implies 
$$(b(x), x)\leq -\frac{\lambda}{2}|x|^2+C.$$ 

\begin{example} \label{examp-3.1}
Let $b \in C^{1,\delta}_{\text{loc}}$ for some $\delta \in (0, 1)$ be eventually strictly monotone and $(L_t)_{t\geq 0}$ be a rotationally invariant $\alpha$-stable L\'evy process~\cite{Sato}. Consider the SDE
\[
    \mathrm{d}X_t = b(X_t) \mathrm{d}t + \sigma \mathrm{d}L_t \quad \text{on } \mathbb{R}^d
\]
with $\sigma > 0$. If $\sigma$ is large enough, then $\lambda_{\text{top}} < 0$.
\end{example}
\begin{proof}
There is a corresponding white noise RDS $\varphi$ with $\varphi_t(\omega, \cdot) \in C_{loc}^{1,\delta}$ and $D\varphi_t(\omega, x)$ satisfies the equation
\[
\frac{d}{dt}D\varphi_t(\omega, x) = Db\left(\varphi_t(\omega, x)\right)D\varphi_t(\omega, x), \quad D\varphi_0(\omega, x) = \text{Id}.
\]
Define
\[
\lambda^{+}(x) := \max_{|r|=1} (Db(x)r, r).
\]
We have
\[
\lambda_\mathrm{top} \leq \liminf_{n \to \infty} \frac{1}{n} \int_{0}^{n} \lambda^{+}(\varphi_s(\omega, x))
\dif s.
\]
Since $b$ satisfies \eqref{eq:drift}, we have that $\lambda^{+}(x) \leq C$ for all $x \in \mathbb{R}^d$ and some constant $C > 0$. Ergodicity and monotone convergence then yield
\[
\lambda_\mathrm{top} \leq \int_{\mathbb{R}^d} \lambda^{+}(x) \dif \rho(x). 
\]
Using \eqref{eq:drift} again, we have
\begin{align}\label{eq:Lya}
\int_{\mathbb{R}^d} \lambda^{+}(x)\dif \rho(x)
&= \int_{B_R} \lambda^{+}(x) \dif \rho(x) + \int_{B_R^c} \lambda^{+}(x) \dif \rho(x)\nonumber \\
&\leq \|Db\|_{C^0(B_R)}\rho(B_R) - \lambda \rho(B_R^c).
\end{align}
Next, we will prove that 
for each $R > 0$, $\rho(B_{R}) \to 0$ for $\sigma \to \infty$. Thus, the right-hand side in \eqref{eq:Lya} becomes negative for $\sigma$ large enough, which finishes the proof.

For every $\alpha \in (0,2)$, we denote by $\mathcal{L}^{(\alpha)}$ the
Kolmogorov operator.
Then, we have for every $f\in C_b^2$ that
\begin{align}\label{Lf-alpha}
\mathcal{L}^{(\alpha)} f(x) =&  b(x)\cdot\nabla f(x) +
\int_{\mathbb{R}^d}(f(x+ \sigma z) - f(x) - I_{|z|\leq 1}\sigma z\cdot \nabla f(x)) \nu^{(\alpha)} (\dif z) \nonumber\\
=&  b(x)\cdot\nabla f(x) +|\sigma|^\alpha \Delta^{\frac{\alpha}{2}}f(x).
\end{align}
Where 
$$
\Delta^{\frac{\alpha}{2}} f(\cdot) =-(-\Delta)^{\frac{\alpha}{2}} f(\cdot) =\int_{\mathbb{R}^d}(f(\cdot+ z) - f(\cdot) - I_{|z|\leq 1} ~z\cdot \nabla f(\cdot)) \nu^{(\alpha)} (\dif z) 
$$
and
$$
\nu^{(\alpha)} (\dif z) := C(d,\alpha)\dif z/|z|^{d+\alpha}, \quad C(d,\alpha):= \frac{\alpha\Gamma(\frac{d+\alpha}{2})
}{2^{2-\alpha}\pi^{\frac{d}{2}}\Gamma(\frac{2-\alpha}{2})}.
$$
Due to the invariance, we have
\begin{equation} \label{inv-cond}
\int  \mathcal{L}^{(\alpha)} f(x) \rho(\dif x)=0.
\end{equation}
Thus
\begin{equation}\label{eq:Lya1}
\int (-\Delta)^{\frac{\alpha}{2}} f(x) \rho (\dif x)
\leq |\sigma|^{-\alpha}  \int  |b(x)\cdot\nabla f(x)|  \rho(\dif x) 
\leq  |\sigma|^{-\alpha}  |b\cdot\nabla f|_{C^0} \rightarrow 0 \ \  \rm{as} \ \ \sigma \rightarrow \infty.
\end{equation}
 Let $p^{(\alpha)}(t,x)$ be the smooth probability density of the $d$-dimensional rotationally invariant $\alpha$-stable process with $\alpha \in (0,2]$.  In particular,
\begin{equation}\label{eq:3101}
p^{(\alpha)}(t,x) = t^{-d/\alpha}p^{(\alpha)}(1,t^{-1/\alpha}x),\ \  t>0, x \in \R^d.
\end{equation}
By \cite[Eqs.(2.8),(2.11)]{CZ16},  for each $\alpha\in (0,2)$, there are  constants $c_i=c_i(\alpha, d), i=1,2,$ such that
\begin{equation}\label{eq:3102}
  \frac{c_1t}{(t^{1/\alpha}+|x|)^{d+\alpha}} \leq p^{(\alpha)}(t,x)\leq \frac{c_2t}{(t^{1/\alpha}+|x|)^{d+\alpha}}.
\end{equation}
Moreover,  there is a constant $c_3=c_3(\alpha, d)$ such that
\begin{equation}\label{eq:3103}
|\nabla p^{(\alpha)}(t,x)|\leq  \frac{c_3t}{(t^{1/\alpha}+|x|)^{d+\alpha+1}}. 
\end{equation}
Let $u=t^{-1/\alpha}x$.  Using~\eqref{eq:3101}, we have
\begin{equation}\label{eq:3104}
(-\Delta)^{\frac{\alpha}{2}}p^{(\alpha)}(t,x)=- \frac{\partial p^{(\alpha)}(t,x)}{\partial t}=\frac{1}{\alpha} t^{-\frac{d}{\alpha} - 1} \left[ d  p^{(\alpha)}(1,u) + u \cdot \nabla p^{(\alpha)}(1,u) \right].
\end{equation}
By~\eqref{eq:3102},~\eqref{eq:3103} and~\eqref{eq:3104}, there is a  constant $\varepsilon=\varepsilon(\alpha,d)\in(0,1 \wedge \frac{c_1 d}{c_3})$ such that  for all $|u|=|t^{-1/\alpha}x|\leq \varepsilon$,
$$
d  p^{(\alpha)}(1,u) + u \cdot \nabla p^{(\alpha)}(1,u) \geq \frac{d}{2}  p^{(\alpha)}(1,u) \geq \frac{c_1 d}{2(1+\varepsilon)^{d+\alpha}},
$$
and for all $t\geq (|x|/\varepsilon)^\alpha$,
\begin{equation}\label{eq:3105}
(-\Delta)^{\frac{\alpha}{2}}p^{(\alpha)}(t,x) \geq \frac{d}{2\alpha} t^{-\frac{d}{\alpha} - 1} \frac{c_1}{(1+\varepsilon)^{d+\alpha}}.
\end{equation}
Using~\eqref{eq:3102} and ~\eqref{eq:3103}, we have for all $|u|=|t^{-1/\alpha}x|\leq \frac{c_1 d}{c_3}$,
\begin{equation}\label{eq:3106}
d  p^{(\alpha)}(1,u) + u \cdot \nabla p^{(\alpha)}(1,u) > 0.
\end{equation}
For every $R>0$, choosing $f_\varepsilon(x)=(R/\varepsilon)^{d} p^{(\alpha)}((R/\varepsilon)^\alpha,x)$ and using~\eqref{eq:3105}, we have
\begin{align*}
\rho(B_R)\leq  
 \frac{2\alpha}{c_1 d} (R/\varepsilon)^{\alpha} (1+\varepsilon)^{d+\alpha}
 \int_{B_R} (-\Delta)^{\frac{\alpha}{2}} f_\varepsilon(x) \rho (\dif x).
\end{align*} 
By \eqref{eq:3106}, we get
$$
\rho(B_R)\leq \frac{2\alpha}{c_1 d} (R/\varepsilon)^{\alpha} (1+\varepsilon)^{d+\alpha}
 \int_{B_{R_\varepsilon}} (-\Delta)^{\frac{\alpha}{2}} f_\varepsilon(x) \rho (\dif x),
$$
where $R_\varepsilon= \frac{c_1 d}{c_3} \frac{R}{\varepsilon} >R$.
Thank to \eqref{eq:Lya1}, we get 
\begin{align*}
\rho(B_R)
\leq \frac{2\alpha}{c_1 d} (R/\varepsilon)^{\alpha } (1+\varepsilon)^{d+\alpha}\left[|\sigma|^{-\alpha}  |b\cdot\nabla f_\varepsilon|_{C^0}+\left|\int_{B^c_{R_\varepsilon}} (-\Delta)^{\frac{\alpha}{2}} f_\varepsilon(x) \rho (\dif x)\right| \right]. 
\end{align*} 
Using~\eqref{eq:3104} again, we obtain
$$
\frac{2\alpha}{c_1 d} (R/\varepsilon)^{\alpha} (1+\varepsilon)^{d+\alpha}|(-\Delta)^{\frac{\alpha}{2}} f_\varepsilon(x)| \leq c_4
$$
for some constant $c_4$ which is independent of the choice of $\varepsilon\in(0,1) $. Consequently,
\begin{align*}
\rho(B_R) \leq 
& \frac{2\alpha}{c_1 d} (R/\varepsilon)^{\alpha} (1+\varepsilon)^{d+\alpha}{|\sigma|^{-\alpha}  |b\cdot\nabla f_\varepsilon|_{C^0}}   + c_4 \rho(B^c_{{R_\varepsilon} }).
\end{align*} 
Notice that $R_\varepsilon \rightarrow \infty$ as $\varepsilon \rightarrow 0$.
Letting first $\sigma \rightarrow \infty$ and then $\varepsilon \rightarrow 0$, 
we obtain that $\rho(B_R) \rightarrow 0$.
This completes the proof.
\qed
\end{proof}

In general, estimating the top Lyapunov exponent is an inherently arduous task in the study of chaotic dynamical systems. The following result concerns the sum of Lyapunov exponents. 
\begin{example} \label{examp-3.2}
Consider the SDE in $\R^d$
\[
    \dif X_t = b(X_t) \dif t + \sigma \dif L_t  
\]
Let  $b\in C^{1,\delta}_{\text{loc}}$ for some $\delta \in (0, 1)$
and $(L_t)_{t\geq 0}$ be a rotationally invariant $\alpha$-stable L\'evy process.  There is a corresponding white noise RDS $\varphi$ with $\varphi_t(\omega, \cdot) \in C_{loc}^{1,\delta}$, and $D\varphi_t(\omega, x)$ satisfies the equation
\[
\frac{d}{dt}D\varphi_t(\omega, x) = Db\left(\varphi_t(\omega, x)\right)D\varphi_t(\omega, x), \quad D\varphi_0(\omega, x) = \text{Id},
\]
and define a linear cycle $\psi$.
 Assume that $\phi$ admits a unique invariant measure $\rho(\dif x)=p(x)\dif x$ with strictly positive and smooth density $p$, and the linear cycle $\psi$ satisfies the integrability conditions~\eqref{eq:IC}. 
 Then the term (which is related to the sum of Lyapunov exponents)
 $$
 \lim_{t\rightarrow \infty} \frac{1}{t}\int_0^t {\rm trace}(Db\left(\varphi_s(\omega, \cdot)\right)
 )\dif s <0.
 $$
In particular, in the one-dimensional case, the top Lyapunov exponent is negative.
\end{example}

\begin{proof} 
The linear cycle $\psi$ satisfies the integrability conditions. According to the multiplicative ergodic theorem, the Lyapunov exponent associated with the invariant measure $\mu = \delta_{\alpha(\omega)}$ exists as a limit. Further, it can be computed according to the Birkhoff-Chintchin ergodic theorem. That is
\[
\lim_{t\rightarrow \infty} \frac{1}{t}\int_0^t {\rm trace}(Db\left(\varphi_s(\omega, \cdot)\right)
 )\dif s =  \int_{\R^d} \nabla\cdot b(x) \rho(\dif x) = \int_{\R^d} \nabla\cdot b(x) p(x) \dif x,
\]
where  $p(x)$ is the density of the Markov invariant measure. Let $f = \ln p$ based on the strictly positive property of $p$.
By It\^o's formula, we get
\begin{align*}
\ln p(X_t) - \ln p(X_0) &= \int_0^t \frac{b(X_{s-}) \cdot\nabla p(X_s-)}{p(X_s-)}  \dif s - \int_0^t |\sigma| ^\alpha  (-\Delta)^{\frac{\alpha}{2}}\ln p(X_s) \dif s +M_t. 
\end{align*}
where $M_t$ is a martingale. 
Let $X_t$ be the stationary solution, i.e., the solution with initial data distributed according to the invariant measure. Then, taking the expectation on both sides, we obtain that
\[
\int_{\mathbb{R}^d} \frac{b \cdot \nabla p}{p}   \rho(\dif x) - \int_{\R^d} |\sigma| ^\alpha (-\Delta)^{\frac{\alpha}{2}} \ln p(x) \rho(\dif x)=0,
\]
i.e.
\begin{equation} \label{eq:FK1}  
 \int_{\mathbb{R}^d} b \cdot \nabla p  \dif x - \int_{\R^d} |\sigma| ^\alpha (-\Delta)^{\frac{\alpha}{2}}\ln p(x) ~ p(x)\dif x =0.
\end{equation}
On the other hand, by the stationary Fokker-Planck equation, we can also get
\[
-\int_{\mathbb{R}^d} \nabla\cdot (b p) \dif x - \int_{\R^d} |\sigma| ^\alpha (-\Delta)^{\frac{\alpha}{2}} p(x) \dif x= 0.
\]
That is
\begin{equation} \label{eq:FK2}  
-\int_{\mathbb{R}^d} p \nabla\cdot b \dif x -\int_{\mathbb{R}^d} b \cdot \nabla p \dif x - \int_{\R^d} |\sigma| ^\alpha (-\Delta)^{\frac{\alpha}{2}} p(x) \dif x= 0.
\end{equation}
Using \eqref{eq:FK1} and \eqref{eq:FK2}, we have
\begin{align*}
 \lim_{t\rightarrow \infty} \frac{1}{t}\int_0^t {\rm trace}(Db\left(\varphi_s(\omega, \cdot)\right)
 )\dif s &= \int_{\mathbb{R}^d}  p~\nabla\cdot b \dif x \\
&= - |\sigma| ^\alpha\int_{\mathbb{R}^d} (-\Delta)^{\frac{\alpha}{2}} \ln p \cdot p \dif x - |\sigma| ^\alpha\int_{\mathbb{R}^d} (-\Delta)^{\frac{\alpha}{2}} p \dif x \\
&= - |\sigma| ^\alpha\int_{\mathbb{R}^d} (-\Delta)^{\frac{\alpha}{4}} \ln p \cdot (-\Delta)^{\frac{\alpha}{4}} p \dif x \\
&< 0.
\end{align*}
The last step involves noticing the monotonicity of the logarithmic function and the definition of $(-\Delta)^{\frac{\alpha}{4}}$. 
\qed
\end{proof}

\begin{remark}
Let  $b=-\nabla V$ with $V\in C^{2,\delta}_{\text{loc}}$ and $L$ be the Brownian motion. Assume that $e^{-\frac{2V(\cdot)}{\sigma^2}} \in L^1(\R^d)$, the unique invariant measure is given by 
$$
\rho(\dif x)=p(x)\dif x=\frac{1} {C_\sigma} e^{-\frac{2V(x)}{\sigma^2}} \dif x,
$$
where $C_\sigma=\int e^{-\frac{2V(x)}{\sigma^2}} \dif x$.  The integrability conditions are detectable.
\end{remark}

\begin{remark} 
{ The integrability condition often fails for SDEs driven by an arbitrary L\'evy noise.  The multiplicative ergodic theorem is not applicable in this case.  In particular, an SDE driven by an arbitrary L\'evy process (beyond a stable process) may not admit a transition probability density.  
On the other hand, it seems that a negative top Lyapunov exponent is not necessary for synchronization.} 
\end{remark}

The following result demonstrates local asymptotic stability without the presence of a negative top Lyapunov exponent.

\begin{example}\label{exam:3.3}  Consider an SDE
   $$ \dif X_t= (X_t-X_t^3)  \dif t+ \sigma \dif S_t,\  \sigma \neq 0$$
where $S_t$ is a subordinator, i.e. a one-dimensional non-decreasing L\'evy process~\cite{Applebaum}. 
Let $\phi_t(\omega,x):=X_t(x)$. Then $\phi$ is asymptotically stable on $B(1,\frac{1}{4})$ $(B(-1,\frac{1}{4}))$ for $\sigma>0$ $(\sigma<0)$.
\end{example}

\begin{proof} 
If $\sigma>0$,  we show that $\phi$ is asymptotically stable on $B(1,\varepsilon)$ with $0<\varepsilon\leq \frac{1}{4}$.
Let $\psi_t(x)$ be the solution map of $\dot{u}=u-u^3,u_0=x$. We have for any $x\in \R$, $\psi_t(x)\leq\phi_t(\omega,x)$.  For all $x,y\in B(1,\varepsilon)$, we observe that $\phi_t(\omega,x),\phi_t(\omega,y)>0$ and 
\begin{align*}
 \ddt(\phi_t(\omega,x)-\phi_t(\omega,y))^2
 &=2(\phi_t(\omega,x)-\phi_t(\omega,y))^2(1-\phi_t^2(\omega,x)-\phi_t(\omega,x)\phi_t(\omega,y)-\phi_t^2(\omega,y)) \\  
&\leq 2(\phi_t(\omega,x)-\phi_t(\omega,y))^2(1-\psi_t^2(x)-\psi_t(x)\psi_t(y)-\psi_t^2(y)). 
\end{align*}
Without loss of generality, we assume that $x <y$ and get
\begin{align*}
\ddt(\phi_t(\omega,x)-\phi_t(\omega,y))^2
&\leq 2(\phi_t(\omega,x)-\phi_t(\omega,y))^2(1-3\psi_t^2(x))\\
&\leq 2(\phi_t(\omega,x)-\phi_t(\omega,y))^2(1-3\psi_t^2(1-\varepsilon))\\
&\leq 2(\phi_t(\omega,x)-\phi_t(\omega,y))^2(1-3 (1-\varepsilon)^2),
\end{align*}
with $1-3 (1-\varepsilon)^2<0$.
Due to monotonicity, we have
$$
{\rm{diam}}(\phi_{t}(\cdot, B(1,\varepsilon)))=\phi_t(\omega,1+\varepsilon)-\phi_t(\omega,1-\varepsilon).
$$
Consequently
$$
\PX \left( \lim_{t\rightarrow \infty }{\rm{diam}}(\phi_{t}(\cdot, B(1,\varepsilon))) =0\right) =\PX \left( \lim_{t\rightarrow \infty }|\phi_t(\omega,1+\varepsilon)-\phi_t(\omega,1-\varepsilon)| =0\right)=1>0.
$$
That is $\phi$ is asymptotically stable on $B(1,\varepsilon)$.
If $\sigma<0$, a similar argument shows that $\phi$ is asymptotically stable on $B(-1,\varepsilon)$. 
\qed
\end{proof}


\subsection{\bf Global and local swift transitivity}  \label{S32}


Consider the following SDE
\begin{equation}\label{eq:3.2.1}
\mathrm{d}X_t =  b(X_t) \dif t+ \sigma \dif L_t,
\end{equation}
where $(L_t)_{t \geq 0}$ is a L\'evy process on $\mathbb{R}^d$ and $\sigma \neq 0$.  Assume that the drift term $b$ is regular enough, and there is a corresponding RDS $\phi$.  Thanks to the smooth effect of white Gaussian noise,  the global swift transitivity property holds if the noise is a Brownian motion (see \cite[Proposition 3.10]{FGS17}).  As mentioned before, this is not true for SDEs driven by general L\'evy noise or degenerate noise.  However, for a class of  L\'evy processes, we derive a result on the global swift transitivity property.

We say that a vector field $b: \mathbb{R}^d \to \mathbb{R}^d$ satisfies the one-sided Lipschitz condition if  
\[
(b(x) - b(y), x - y) \leq \lambda |x - y|^2 
\]
for all $x,y\in \R^d$ and some $\lambda>0$.

\begin{lemma}\label{lem:3.2.1} 
Let $b$ be locally Lipschitz and satisfy the one-sided Lipschitz condition. 
Assume that for every $T>0$, $L$ has full support in $D([0,T],\R^d)$. 
Then for every $x,y\in E$ and $R>0$, there is a time $t_0>0$ such that  
$$
\PX \left( \phi_{t_0}(\cdot,B(x,R)) \subset B (y, 2 R) \right)>0.
$$
In particular, the global swift transitivity property holds.
\end{lemma}

\begin{proof}
Set $t_0 = \frac{1}{\lambda}\ln \frac{3}{2}$.  For $t \in [0, t_0]$ define
\[
\psi(t) := x + \frac{t}{t_0}(y - x)
\]
 and
\[
\omega^0(t) := \frac{1}{\sigma} \left( \psi(t) - x - \int_0^t b(\psi(s)) ds \right).
\]
 Then $\varphi_t(\omega^0, x) = \psi(t)$ for all $t \in [0, t_0]$. In particular, $\varphi_{t_0}(\omega^0, x) = y$. By the one-sided Lipschitz condition of the drift term, we have 
\begin{align*}
\ddt |\phi_t(\omega, x')-\phi_t(\omega, x)|^2  
&=  2 ( b(\phi_t(\omega, x')) - b(\phi_t(\omega, x)), \phi_t(\omega, x') - \phi_t(\omega, x) )\\
&\leq 2 \lambda |\phi_t(\omega, x') - \phi_t(\omega, x)|^2.
\end{align*}
For all $x' \in B(x, r)$, $\omega \in \Omega$, and $t \geq 0$. By the Gronwall inequality, it follows that
\[
|\varphi_t(\omega, x') - \varphi_t(\omega, x)| \leq |x' - x| e^{\lambda t} \leq r e^{\lambda t}. 
\]
Consequently
\begin{align*}
|\varphi_{t_0}(\omega, x') - y| 
&\leq |\varphi_{t_0}(\omega, x') - \varphi_{t_0}(\omega, x)| + |\varphi_{t_0}(\omega, x) - \varphi_{t_0}(\omega^0, x)|\\
&\leq \frac{3}{2} r + |\varphi_{t_0}(\omega, x) - \varphi_{t_0}(\omega^0, x)|.
\end{align*}
The map $\omega \mapsto \varphi_{t_0}(\omega, x)$ is continuous from $D_U([0, t_0]; \mathbb{R}^n)$ to $\mathbb{R}^d$. Then there exists some $\delta > 0$ such that
\begin{align*}
\PX(\varphi_{t_0}(\cdot, B(x, r)) \subset B(y, 2r)) 
&\geq \PX \left( |\varphi_{t_0}(\cdot, x) - \varphi_{t_0}(\omega^0, x)| \leq \frac{r}{2} \right)  \\
&\geq \PX \left( \sup_{s \in [0, t_0]} |\omega(s) - \omega^0(s)| \leq \delta \right) > 0.
\end{align*}
Here, in the last inequality, we used the full support property of L\'evy processes.
\qed\end{proof}

\begin{remark}
The global strong swift transitivity property does not hold for SDEs only with one-sided Lipschitz conditions. For example $\ddt X=X+\dot{W}.$   
{Following the proof of Lemma~\ref{lem:3.2.1}, we can not achieve the global strong swift transitivity property even with the one-sided Lipschitz condition replaced by eventual strict monotonicity~\eqref{eq:drift}. }

As pointed out in Remark~\ref{rem:2.4}, the global (strong) swift transitivity property does not hold for SDEs driven by general L\'evy noise or degenerate noise.  Consequently, we turn to verifying the local strong recurrence property under even weaker drift conditions. 
\end{remark}
Assume that   there exists some $z \in \mathbb{R}^d$ and $\lambda >0$  such that  for all $x \in\R^d$ 
\begin{equation}\label{eq:drift-weak}
(b(x) - b(z), x - z) \leq -\lambda |x - z|^2. 
\end{equation}
\begin{remark}
The condition \eqref{eq:drift-weak} is essentially weaker than \eqref{eq:drift}.
{ The drift term of the Lorentz 63 system (with the Rayleigh number $\rho<1$) satisfies the condition \eqref{eq:drift-weak}, but not \eqref{eq:drift}.} 
\end{remark}

\begin{lemma}\label{lem:3.2.2} 
Assume that $b$ is locally Lipschitz and satisfies \eqref{eq:drift-weak}. 
For every $T>0$, $L$ has full support in $D([0,T],\R^d)$.
Then $\phi$ is strongly recurrent at $z$.
\end{lemma}

\begin{proof}
For fixed $z$ and any $R>0$, define 
\[
\psi(t) := z,\  \omega^0(t) := \frac{t}{\sigma} b(z).
\]
Then $\phi_t(\omega^0, z) = z$ for all $t \geq 0$.  
We have 
\[
( b(z) - b(y), z - y ) \leq -\lambda|z - y|^2
\] for all $y \in \mathbb{R}^d$. This inequality and $\phi_t(\omega^0, z) = z$ imply
\begin{align*}
\frac{d}{dt} |\phi_t(\omega^0, z) - \phi_t(\omega^0, y)|^2 
&= 2 ( b(\phi_t(\omega^0, z)) - b(\phi_t(\omega^0, y)), \phi_t(\omega^0, z) -\phi_t(\omega^0, y) )\\ 
&\leq -2\lambda |\phi_t(\omega^0, z) - \phi_t(\omega^0, y)|^2.
\end{align*}
Using Gronwall's inequality, it follows that, for all $y \in B(z, R)$ and $t \geq 0$, 
\[
|z- \phi_t(\omega^0, y)| \leq |z - y| e^{-2\lambda t} \leq R e^{-2\lambda t}.
\]
Choose $t_0 \geq 0$ such that $e^{-2t_0} < \frac{1}{4}$. Then for all $y \in B(z, R)$ and $\omega \in \Omega$,
\begin{align*}
|z - \phi_{t_0}(\omega, y)| 
&\leq |z - \phi_{t_0}(\omega^0, y)| + |\phi_{t_0}(\omega^0, y) - \phi_{t_0}(\omega, y)| \\
&\leq \frac{R}{4} + |\phi_{t_0}(\omega^0, y) - \phi_{t_0}(\omega, y)|.
\end{align*}
The map $\omega \mapsto \phi_{t_0}(\omega, \cdot)$ is continuous from $D([0, t_0]; \mathbb{R}^n)$ to $C(B(z, R); \mathbb{R}^d)$. Then there exists some $\delta > 0$ such that
\begin{align*}
\PX \left( \phi_{t_0}(\cdot, B(z, R)) \subset B \left( z, \frac{R}{2} \right) \right) 
&\geq \PX \left( \sup_{y \in B(x, R)} |\phi_{t_0}(\omega^0, y) - \phi_{t_0}(\omega, y)| \leq \frac{R}{4} \right)\\
&\geq \PX \left( \sup_{s \in [0, t_0]} |\omega(s) - \omega^0(s)| \leq \delta \right) > 0.
\end{align*}
The proof is complete.
\qed\end{proof}

\begin{remark}\label{rem:3.5}
Although Lemma~\ref{lem:3.2.2} is established for non-degenerate noise, the proof can be applied to some SDEs with degenerate noise.   
\end{remark}


\section{\bf Concrete examples }  \label{Sec:example}

In this section, we consider several concrete examples, including an SDE forced by Poisson processes, an SDE with drift given by a multidimensional double-well potential with degenerate additive noise, and stochastic Lorenz systems driven by degenerate noise.  The global swift transitive fails for these examples.


\subsection{\bf An SDE driven by a Poisson process  }  \label{S41}


Consider an SDE driven by Poisson processes
\begin{equation}\label{eq:4-1}
\dif X_t= (X_t-X_t^3) \dif t+ \sigma \dif N_t,   \ \ \sigma>0.
\end{equation}
A Poisson process does not have full support in $D([0, T],\R)$. We can not apply Lemma~\ref{lem:3.2.2} to show the local strong recurrence property. 

\begin{theorem}
Synchronization holds for system~\eqref{eq:4-1}. 
\end{theorem}
\begin{proof} 
Note that $\phi$ is not globally swift transitive. We verify that $\phi$ is strongly recurrent at $1$. Using the monotonicity, we have 
$$
\phi_{t}(\cdot,B(1,R))=\left(\phi_{t}(\cdot,1-R),\ \phi_{t}(\cdot,1+R) \right)
$$
and
$$
\left( \phi_{t}(\cdot,B(1,R)) \subset B \bigg(1, \frac{R}{2}\bigg) \right)
=\left(\phi_{t}(\cdot,1-R)> 1-\frac{R}{2},\ \phi_{t}(\cdot,1+R)<1+\frac{R}{2} \right).
$$
{For any $T>0$, $n\in\N^+$, $0<s_1<\cdots<s_n<T$, 
$$
\PX(N_{s_1}(\cdot)=0, N_{s_2}(\cdot)-N_{s_1}(\cdot)=1, \cdots, N_{s_n}(\cdot)-N_{s_{n-1}}(\cdot)=1 )>0.
$$}
From this and $\ddt(\phi_t(\omega,x)-\phi_t(\omega,y))^2\leq 2(\phi_t(\omega,x)-\phi_t(\omega,y))^2$, we have there exists $t_1=t_1(R)\in (0,\ln 2)$ such that
$$
\PX(\phi_{t_1}(\cdot,B(1,R)) \subset   (1, 1+4R)))>0.
$$
Notice
\begin{align*}
&\PX \left( \phi_{t_1+t_2}(\cdot,B(1,R)) \subset B \bigg(1, \frac{R}{2}\bigg) \right)\\
&\geq \PX \left( \phi_{t_1+t_2}(\cdot,B(1,R)) \subset B \bigg(1, \frac{R}{2}\bigg), \phi_{t_1}(\cdot,B(1,R)) 
  \subset   (1, 1+4R)) \right)\\
& \geq \PX \left( \phi_{t_2}(\theta_{t_1}\cdot,(1,1+4R)) \subset B \bigg(1, \frac{R}{2}\bigg), \phi_{t_1}(\cdot,B(1,R)) \subset   (1, 1+4R)) \right).  
\end{align*}
Since $\mathcal{F}_{0,t_1}$  and $\mathcal{F}_{t_1,t_1+t_2}$ are independent, we get
\begin{align*}
  &\PX \left( \phi_{t_1+t_2}(\cdot,B(1,R)) \subset B \bigg(1, \frac{R}{2}\bigg) \right)\\
  &\geq \PX \left( \phi_{t_2}(\theta_{t_1}\cdot,(1,1+4R)) \subset B \bigg(1, \frac{R}{2}\bigg) \right) \PX \left( \phi_{t_1}(\cdot,B(1,R)) \subset   (1, 1+4R)) \right).  
\end{align*}
Since $\PX$ is invariant under $\theta_t$, we have 
\begin{align*}
&\PX \left( \phi_{t_1+t_2}(\cdot,B(1,R)) \subset B \bigg(1, \frac{R}{2}\bigg) \right)\\
&\geq \PX \left( \phi_{t_2}(\cdot,(1,1+4R)) \subset B \bigg(1, \frac{R}{2}\bigg) \right) \PX \left( \phi_{t_1}(\cdot,B(1,R)) \subset   (1, 1+4R)) \right). 
\end{align*}
Similar to the proof in Example~\ref{exam:3.3} and using the monotonicity again, we obtain that
$$
\PX \left( \lim_{t\rightarrow \infty }{\rm{diam}}(\phi_{t}(\cdot, (1,1+4R))) =0\right) =1,
$$
and there exists $t_2>0$ such that 
$$
\PX \left( \phi_{t_2}(\cdot,(1,1+4R)) \subset B \bigg(1, \frac{R}{2}\bigg) \right)>0.
$$
Thus
$$
\PX \left( \phi_{t_1+t_2}(\cdot,B(1,R)) \subset B \bigg(1, \frac{R}{2}\bigg) \right)>0.
$$

By the result in Example~\ref{exam:3.3},  we have $\phi$ is asymptotically stable on $B(1,\frac{1}{4})$.   
Using Theorem~\ref{thm:2.1}, the proof is complete.
\qed\end{proof}

\begin{remark}
As a corollary,  \eqref{eq:4-1} admits a unique invariant measure which is strongly mixing.     
\end{remark}


\subsection{\bf An SDE driven by degenerate noise  }  \label{S42}


Consider the SDE with drift given by a multidimensional double-well potential with degenerate additive noise~\cite{Vorkastner18}. That is
\begin{equation}\label{eq:exp 4-2}
\begin{aligned}
\dif X_t &= \left(X_t - \left|\left(\begin{array}{c} X_t \\ Y_t \end{array}\right)\right|^2 X_t\right) \dif t + \sigma \dif W_t, \quad \text{on} \ \mathbb{R}^n, \\
\dif Y_t &= \left(Y_t - \left|\left(\begin{array}{c} X_t \\ Y_t \end{array}\right)\right|^2 Y_t\right) \dif t, \quad \text{on} \ \mathbb{R}^{d-n},
\end{aligned}
\end{equation}
for $\sigma > 0$ and $d, n \in \mathbb{N}$ with $n < d$. $W_t$ is an $n$-dimensional Brownian motion. 

The following result is from~\cite{Vorkastner18}.
\begin{prop}
In the case $n = 1$ with $\sigma\geq 2$ and in the case $n \geq 2$, synchronization holds for system~\eqref{eq:exp 4-2}.
There is no weak synchronization, in the case $n = 1$ with $\sigma<\frac{1}{2}$.
\end{prop}

This example has been studied in~\cite{Vorkastner18}. 
The random dynamical system associated with \eqref{eq:exp 4-2} is not globally swift transitive. This can be seen by observing that the set $\{(x_1, x_2, \ldots, x_d) \in \mathbb{R}^d: x_i > 0\}$ is not reachable if one starts in $\{(x_1, x_2, \ldots, x_d) \in \mathbb{R}^d: x_i < 0\}$ for some $n < i \leq d$.  

In the case $n = 1$ with $\sigma\geq 2$ and in the case $n \geq 2$, 
the author derive a locally swift transitivity property on $M := \{(x_1, x_2, \ldots, x_d) \in \mathbb{R}^d: x_i = 0 \text{ for } i> n\}$ and contracting on large sets~\cite[Proposition 5.1, Corollary 5.1]{Vorkastner18}.  This, together with local asymptotic stability, implies synchronization. 
{ In fact, the random dynamical system associated with \eqref{eq:exp 4-2} is (locally) strongly recurrent at the origin. Our results, 
Theorem~\ref{thm:2.1}, can be applied directly. This provides a slightly simpler proof.}

On the other hand,   there is no weak synchronization, in the case $n = 1$ with $\sigma<\frac{1}{2}$~\cite[Theorem 6.1]{Vorkastner18}. A bifurcation appears.


\subsection{\bf Stochastic Lorenz 63 systems }  \label{S43}


We consider the stochastic Lorenz 63 system
\begin{equation}\label{eq:CH21}
\begin{aligned}
    \dif X&= \sigma(Y-X)\dif t,  \\
    \dif Y&= (\rho X-Y-XZ) \dif t,\\
    \dif Z&= (-\beta Z+XY) \dif t+ \gamma \dif W_t.   
\end{aligned}
\end{equation} 
When the noise is absent, the system reduces to the famous Lorenz 63 system~\cite{Temam}.
We fix two parameters $\sigma,\beta>0$. Then, the determined system has a global attractor. 
For $\rho<1$, the origin is a hyperbolic sink and is the only attractor. 
At $\rho=1$, a pitchfork bifurcation occurs. 
{ For $\rho>1$, the origin is a saddle point. Two fixed points appear at $(\pm\sqrt{\beta(\rho-1)},\pm\sqrt{\beta(\rho-1)},\rho-1)$. At $\rho=\rho_h=\frac{\sigma(\sigma+\beta+3)}{\sigma-\beta-1}$, a Hopf bifurcation occurs at the nontrivial fixed points.
The nontrivial fixed points are stable for $1<\rho <\rho_h$. The nontrivial fixed points are unstable when $\rho >\rho_h$.
}

Recently, M. Coti~Zelati and M. Hairer~\cite{CH21} studied the invariant measure of system~\eqref{eq:CH21}. We recall their main result.
\begin{prop} \label{prop:4.3.1}
For any $\sigma, \beta > 0$ and any $\rho < 1$, there exist values $0 < \gamma_* \leq \gamma^* < \infty$ such that
\begin{enumerate}
    \item For $|\gamma| < \gamma_*$, \eqref{eq:CH21} admits $\nu_1=\delta_0\times\delta_0\times N(0,\frac{\gamma^2}{2\beta})$ as its unique invariant measure.
    \item For $|\gamma|> \gamma^*$, \eqref{eq:CH21} admits exactly two ergodic invariant measures: $\nu_1$ and another measure $\nu_2$. Furthermore, $\nu_2$ has a smooth density with respect to Lebesgue measure on $\R^3 \setminus H$. Where  $H:=\{(x,y,z): x = y = 0\}$.
\end{enumerate}
For $\rho \geq 1$, there exists $\gamma^* \geq 0$ such that the second statement still holds.
\end{prop}

\begin{remark} \em
 M. Coti~Zelati and M. Hairer detect a noise-induced transition for the invariant measure of the Lorenz system~\eqref{eq:CH21}.   Although their results do indeed guarantee that, in this case of $\rho < 1,|\gamma| < \gamma_*$, the system admits a unique invariant measure,  they provide little information about the specific value of $\gamma_*,\gamma^*$ and
 whether $\nu_1$ is strong mixing. Consequently, they say {\em ``One would naturally expect to have  $\gamma_* = \gamma^* $, but we cannot guarantee this at the moment. $\cdots$ One would
also expect to have $\gamma^*=0$  when $\rho > 1$  $\cdots$,  we cannot rule out the existence of an intermediate range of values (of $\gamma$) for which $\nu_1$ would be the unique invariant measure."}
 \end{remark}

Let  $O^\lambda _t(\omega)=O^\lambda (\theta_t\omega)$ be the  stationary solution of the following
Langevin equation
 \begin{equation*}
      \dif O^\lambda = -\lambda O^\lambda  dt +
     \gamma \dif W_t.
 \end{equation*}
Here $\lambda >0$ is a  parameter to be determined, and 
\[
O^{\lambda}(\omega)= \int_{-\infty}^0
 \gamma e^{\lambda s} \dif W_s=-\int_{-\infty}^0
\gamma \lambda e^{\lambda s}\omega(s)\dif s, \; \omega
\in \Omega.
\]
Define $z_t(\omega)=Z_t(\omega)-O^\lambda_t (\omega)$, we have
\begin{equation}\label{eq:CH21-1}
  \begin{aligned}
    \dif x&=\sigma(y-x)\dif t,  \\
    \dif y&= ((\rho-O^\lambda ) x-y-xz) \dif t,\\
    \dif z&= (-\beta z+xy-(\beta-\lambda)O^\lambda) \dif t. 
  \end{aligned}  
\end{equation}

The following lemma derives the unique random attractor of \eqref{eq:CH21-1}.  As a result, there is a unique global attractor of \eqref{eq:CH21}.
\begin{lemma} \label{lem:4.3.1} For any $\gamma, \rho\in \R, \beta,\sigma>0$,  there exists a unique global attractor of \eqref{eq:CH21-1}. 
\end{lemma}

\begin{proof}
Let $L:=\frac{1}{2}(x^2+y^2+(z-\rho-\sigma)^2)$ and write $O=O^\lambda$ for simiplicity. We have that 
$$
\ddt L=-\sigma x^2-y^2-\beta z^2-O xy+\beta (\rho+\sigma) z +(z-\rho-\sigma)(\lambda-\beta)O.
$$ 
By the Young inequality and $-z(z-\rho-\sigma)=-\frac{1}{2} z^2-\frac{1}{2}(z-\rho-\sigma)^2+\frac{1}{2}(\rho+\sigma)^2$, we have
\begin{align*}
\ddt L &\leq -\sigma x^2-y^2+\frac{\sigma}{2}x^2+ \frac{O^2}{2\sigma} y^2-\frac{\beta}{2} (z-\rho-\sigma)^2  +\frac{\beta}{2}(\rho+\sigma)^2\\
 &\quad+\frac{\beta}{4} (z-\rho-\sigma)^2 +\frac{2}{\beta}(\lambda-\beta)^2O^2\\
      &\leq 2 \max\{-\frac{\sigma}{2}, -\frac{\beta}{4}, -1+ \frac{O^2}{2\sigma}\} L +\frac{\beta}{2}(\rho+\sigma)^2+\frac{2}{\beta}(\lambda-\beta)^2O^2.   
\end{align*}
Thus
$$
 \ddt L\leq   (-K+ \frac{O^2}{\sigma}) L+M(\theta_t\omega),   
$$
with $K=\min \{{\sigma}, \frac{\beta}{2}, 2\}$ and $M(\theta_t\omega)=\frac{\beta}{2}(\rho+\sigma)^2+\frac{2}{\beta}(\lambda-\beta)^2O^2(\theta_t\omega)$.
By the Gronwall lemma, we get
$$
    L(s,\omega,L_0)\leq   e^{-K s+ \frac{1}{\sigma}\int_0^s {O^2(\theta_r\omega)} \dif r } L_0+  \int_0^s e^{-K (s-r)+ \frac{1}{\sigma}\int_r^s {O^2(\theta_\tau\omega)} \dif \tau }  M(\theta_r\omega) \dif r.
$$
Replacing $\omega$ by $\theta_{-t}\omega$ and $s$ by $t$, from above inequality we get 
\begin{equation}\label{a}
    L(t,\theta_{-t}\omega,L_0)\leq   e^{-K t+ \frac{1}{\sigma}\int_{-t}^0 {O^2(\theta_r\omega)} \dif r } L_0+  \int_{-t}^0 e^{K r + \frac{1}{\sigma}\int_r^0 {O^2(\theta_r\omega)} \dif r }  M(\theta_r\omega) \dif r.
\end{equation}
By the ergodic theorem, we have 
\begin{equation}\label{b}
    \lim_{t\rightarrow \pm\infty} \frac{1}{t} \int_0^t O^2 (\theta_s \omega) ds = \EX O^2(\omega)=\EX \left(\int_{-\infty}^0
 \gamma e^{\lambda s} \dif W_s \right)^2=\frac{\gamma^2}{2\lambda}. 
\end{equation}
Due to \cite[Proposition 4.3.3]{Arn98}, there exists a tempered function $r(\omega)> 0$ such that
$
|O(\omega)|^2\leq  r(\omega),
$
and for any $\epsilon>0$,
$$
|O(\theta_t\omega)|^2\leq   e^{\epsilon t }r (\omega),
$$ $\mathbb{P}$-a.e. $\omega\in\Omega$.  Thanks to \eqref{b},
we choose $\lambda > \frac{\gamma^2}{\sigma K}$ such that 
\begin{equation}\label{b-2}
 \lim_{t\rightarrow \infty} [-K + \frac{1}{\sigma t}\int_{-t}^0 {O^2(\theta_r\omega)} \dif r] <0, 
\end{equation}
and further
\begin{equation}\label{c}
R(\omega):=\int_{-\infty}^0 e^{K r + \frac{1}{\sigma}\int_r^0 {O^2(\theta_r\omega)} \dif r }  M(\theta_r\omega) \dif r 
\end{equation}
is well-defined.
Due to \eqref{a},\eqref{b-2} and \eqref{c},  there exists a time $T_0(\omega)>0$ and a positive constant $c$ such that for all $t\geq T_0$, 
\begin{equation}\label{d}
     L(t,\theta_{-t}\omega,L_0)\leq  e^{-c t}L_0 +R(\omega).
\end{equation}
By \eqref{d}, we have proved the existence of a compact absorbing set in $\R^3$. The standard method for random attractors implies that there is a global pullback attractor for \eqref{eq:CH21-1}~\cite{Crauel94, Sc97}. \qed
\end{proof}

\begin{lemma}  \label{lem:4.3.2}
Assume that $\rho<1$,  for any $\gamma \in \R, \beta,\sigma>0$,  \eqref{eq:CH21} is strongly recurrent at the origin $z_*=(0,0,0)$.
\end{lemma}
\begin{proof} Note that the drift term of Lorenz system satisfies conditions \eqref{eq:drift-weak} at the origin $z_*$. The proof is similar to that of Lemma~\ref{lem:3.2.2}, although the system is forced by degenerate noise, see Remark~\ref{rem:3.5}. 
Define
\[
 \omega^0(t) := 0, \ \psi_t(\cdot) := \phi_t(\omega^0,\cdot).
\]
Then $\psi_t(\cdot)$ is the solution map of the determined Lorenz 63 system with $\psi_t(z_*)=\phi_t(\omega^0, z_*) = z_*$ for all $t \geq 0$.  
Assume that $\rho<1$. The Lyapunov function method shows that the origin $z_*$ is globally asymptotically stable~\cite{Sparrow12}. Moreover,
for all $y \in B(z_*, R)$, 
\[
\lim_{t\rightarrow \infty}|z_*- \phi_t(\omega^0, y)| =0.
\]
Choose $t_0 \geq 0$ such that for all $y \in B(z_*, R)$
$$
|z_*- \phi_{t_0}(\omega^0, y)|  \leq \frac{R}{16}.
$$
 Then for all $y \in B(z_*, R)$ and $\omega \in \Omega$,
\begin{align*}
 |z_* - \phi_{t_0}(\omega, y)| 
 &\leq |z_* - \phi_{t_0}(\omega^0, y)| + |\phi_{t_0}(\omega^0, y) - \phi_{t_0}(\omega, y)| \\
 &\leq \frac{R}{16} + |\phi_{t_0}(\omega^0, y) - \phi_{t_0}(\omega, y)|.
\end{align*}
The map $\omega \mapsto \phi_{t_0}(\omega, y)$ is continuous from $C([0, t_0]; \mathbb{R})$ to $\mathbb{R}^3$. Then there exists an $\delta > 0$ such that
\begin{align*}
 \PX \left( \phi_{t_0}(\cdot, B(z_*, R)) \subset B \left( z_*, \frac{R}{8} \right) \right) 
 &\geq \mathbb{P} \left( \sup_{y \in B(x, R)} |\phi_{t_0}(\omega^0, y) - \phi_{t_0}(\omega, y)| \leq \frac{R}{16} \right)  \\
 &\geq \mathbb{P} \left( \sup_{s \in [0, t_0]} |\omega(s) - \omega^0(s)| \leq \delta \right) > 0.
\end{align*}
\qed
\end{proof}

\begin{lemma}  \label{lem:4.3.3}
Assume that $|\gamma| < (1-\rho)\sqrt{\pi\beta}$, then 
\eqref{eq:CH21} is locally asymptotically stable. 
\end{lemma}
\begin{proof} 
Let $\phi_1(t)=(X_1(t),Y_1(t),Z_1(t))$ be the solution map of~\eqref{eq:CH21} with initial data $\phi_1(0)$ and $\phi_2(t)=(X_2(t),Y_2(t),Z_2(t))=(0,0,Z_2(t))$ be the solution map of~\eqref{eq:CH21} with initial data $\phi_2(0)=(0,0,0)$.
Then 
$
|\phi_1-\phi_2|^2=(X_1-X_2)^2+(Y_1-Y_2)^2+(Z_1-Z_2)^2=X_1^2+Y_1^2+(Z_1-Z_2)^2.
$
Notice that
\begin{align*}
\ddt  [\frac{X_1^2}{\sigma}+Y_1^2+(Z_1-Z_2)^2]
&\leq -2 X_1^2-2Y_1^2-2\beta(Z_1-Z_2)^2+(2+2\rho-2Z_2)X_1Y_1\\
&\leq -2 X_1^2-2Y_1^2-2\beta(Z_1-Z_2)^2+|1+\rho-Z_2| (X_1^2+Y_1^2)\\
&\leq (|1+\rho-Z_2|-2)( X_1^2+Y_1^2)-2\beta(Z_1-Z_2)^2,   
\end{align*}
where $Z_2$ is the Ornstein-Uhlenbeck process.
By the ergodic theorem, we get
$$
\lim_{t\rightarrow \infty}\frac{1}{t}\int_0^t (|\rho+1-Z_2|-2)\dif s=\EX |\rho+1-O^\beta|-2=
\rho-1+\sqrt{\gamma^2/\pi\beta} <0.
$$
This implies $\lim\limits_{t\rightarrow \infty}|\phi_1(t)-\phi_2(t)|=0$ almost surely, and further implies that $\phi$ is locally asymptotically stable. 
\qed
\end{proof}

\begin{remark}
We give another proof by estimating the Lyapunov exponent.
If $|\gamma|<\gamma_*$, then \eqref{eq:CH21} admits a unique invariant measure $\delta_0\times \delta_0 \times N(0,\frac{\gamma^2}{2\beta})$, and \eqref{eq:CH21-1} with $\lambda=\beta $ admits a unique invariant measure $\delta_0\times \delta_0 \times \delta_0$.   
At $(0,0,0)$, the Jaobian reads 
$$
\begin{pmatrix}
    -\sigma&\sigma&0\\
    \rho+O^\beta&-1&0\\
    0&0&-\beta\\
\end{pmatrix}
$$
Define
$$ D_t=
\begin{pmatrix}
    -\sigma&\sigma\\
    \rho+O_t^\beta&-1\\
\end{pmatrix}.
$$
Let
$$
\ddt(\nabla X, \nabla Y)= D_t(\nabla X, \nabla Y).
$$
We have $\lim\limits_{t\rightarrow \infty}\frac{1}{t}\ln(\frac{\nabla X^2}{\sigma} +\nabla Y^2) \leq \lim\limits_{t\rightarrow \infty} \frac{1}{t}\int_0^t [|\rho+1-O^\beta_s|-2] \dif s $.
The ergodic theorem reads
$$
\lim_{t\rightarrow \infty}\frac{1}{t}\int_0^t (|\rho+1-O^\beta_s|-2) \dif s=\EX |\rho+1-O^\beta|-2=
\rho-1+\sqrt{\gamma^2/\pi\beta}.
$$
Thus, the top Lyapunov exponent is negative if $\gamma^2 < (1-\rho)^2\pi\beta$. Thanks to Proposition~\ref{prop:3.1}, the locally asymptotically stable property holds. 
\end{remark}

\begin{theorem}\label{thm:4.3.1}
If $|\gamma| < (1-\rho)\sqrt{\pi\beta}$,  synchronization occurs.
If $|\gamma|>\gamma^* \geq (1-\rho)\sqrt{\pi\beta}$,  (weak) synchronization does not occur. 
\end{theorem} 
\begin{proof} By Theorem~\ref{thm:2.1} and Lemmas~\ref{lem:4.3.1}-\ref{lem:4.3.3}, we have  synchronization occurs provided $|\gamma| < (1-\rho)\sqrt{\pi\beta}$.

If $|\gamma|>\gamma^*$, using Theorem~\ref{thm:2.3} and Proposition~\ref{prop:4.3.1}, we know that the RDS is not strongly mixing, and weak synchronization does not hold. 
 \qed
\end{proof}

\begin{remark}
 As a corollary, we know that if $|\gamma| < (1-\rho)\sqrt{\pi\beta}$, the unique invariant measure $\nu_1$ is strongly mixing. 
\end{remark}


\section{\bf Proofs }  \label{proof}



\subsection{\bf Proof of Theorem~\ref{thm:2.1} }  \label{proof 2.1}

We begin with some lemmas.

\begin{lemma}\label{lem:5.1}
Assume that synchronization holds. Then
there exist a point $z\in E$  such that  for every  $\varepsilon>0$, $\PX(A\subset B(z,\varepsilon))>0.$

Assume that synchronization holds and $E$ is locally compact and $\sigma$-compact. Then $\phi$ is asymptotically stable on $U$, where $U\subset E$ is an arbitrary non-empty, bounded set.  
\end{lemma}
\begin{proof}
For any $\varepsilon>0$, since $A(\omega)=\{a(\omega)\}$, there exists (indeed for all) $z\in \rm{supp}~ \mu_{a(\omega)} $ such that
$$
\PX\left(a(\cdot)\subset B(z,\varepsilon )\right)>0.
$$

Since $E$ is locally compact and $\sigma$-compact.  As a result, a weak attractor attracts all closed and bounded subsets. Hence
\begin{align*}
\mathrm{diam}(\varphi_t(\omega, U)) &= \sup_{x,y\in U} d(\varphi_t(\omega, x), \varphi_t(\omega, y)) \\
&\leq \sup_{x,y\in U} \left[d(\varphi_t(\omega, x), a(\theta_t\omega)) + d(a(\theta_t\omega), \varphi_t(\omega, y)) \right] \\
&\to 0 \quad \text{for } t \to \infty
\end{align*}
in probability. Thus $\phi$ is asymptotically stable on $U$. In particular,  $\phi$ is asymptotically stable on $B(z,\varepsilon )$.
\qed
\end{proof}

\begin{lemma}  \label{lem:5.2}
 Let $A$ be the unique weak attractor of an {\em RDS} $\phi$. Assume that
there exist a point $z\in E$  and  $\varepsilon>0$ such that  $\phi$ is asymptotically stable on $B(z,\varepsilon)$ and  $\PX(A\subset B(z,\varepsilon))>0.$  Then synchronization occurs.
\end{lemma}
\begin{proof}
For a proof, we refer to \cite[Lemma 2.5]{FGS17}. \qed
\end{proof}

\begin{lemma}\label{lem:5.3}
Assume that synchronization holds and $E$ is locally compact and $\sigma$-compact. Then there exists a point $z\in E$ such that $\phi$ is strongly swift transitive at $z$. 
\end{lemma}
\begin{proof}
For any $R>0$, since $A(\omega)=\{a(\omega)\}$, there exists $z\in E$ such that
$$
\PX\left(a(\cdot)\subset B\bigg(z,\frac{R}{4}\bigg)\right)=2\delta_R>0.
$$
By the definition of weak attractor, for every $x\in E$, there exists a time $t>0$ such that
$$
\PX\left(\phi_{t}(\cdot, B(x,R)) \subset  B\left(a(\theta_t \cdot ),\frac{R}{4}\right)\right)>1-\delta_R.
$$
As consequence 
\begin{align*}
\PX \left( \phi_{t}(\cdot,B(x,R)) \subset B \bigg(z, \frac{R}{2}\bigg) \right)
&\geq \PX \left( a(\theta_t\cdot)\subset B\left(z,\frac{R}{4}\right),\phi_{t}(\cdot,B(x,R)) \subset B \bigg(z, \frac{R}{2}\bigg) \right)\\
&\geq \PX \left( a(\theta_t\cdot)\subset B\left(z,\frac{R}{4}\right),\phi_{t}(\cdot,B(x,R)) \subset B \bigg(a(\theta_t\cdot), \frac{R}{4}\bigg) \right)\\
&\geq \PX \left( a(\theta_t\cdot)\subset B\bigg(z,\frac{R}{4}\bigg)\right)+\PX\left(\phi_{t}(\cdot,B(x,R)) \subset B \bigg(a(\theta_t\cdot), \frac{R}{4}\bigg) \right)-1\\
&> \delta_R. 
\end{align*}
In particular, 
$$
\PX \left( \phi_t(\cdot,B(x,R)) \subset B \bigg(z, \frac{R}{2}\bigg) \right)>0.
$$
\qed\end{proof}

\begin{lemma} \label{lem:5.4}
Let $A$ be the unique weak attractor of an {\em RDS} $\phi$.  If for a point $z\in E$ and every $R>0$, there is a time $t_0 >0$ such that  
$$
\PX \left( \phi_{t_0}(\cdot,B(z,R)) \subset B \bigg(z, \frac{R}{2}\bigg) \right)>0,
$$
then for every $\varepsilon >0$,
$\PX(A\subset B(z,\varepsilon))>0$.
\end{lemma}

\begin{proof} 
Since $A$ is a random compact set, for fixed $z\in E$, we can choose $r_0$ such that 
$$
\PX(A\subset B(z,r_0))>0.
$$
Since $\phi$ is  strongly recurrent at $z$,  there exists $t_0$ such that 
\begin{align*}
\PX(A\subset B(z,\frac{r_0}{2}))
&=\PX(A(\theta_{t_0}\cdot)\subset B(z,\frac{r_0}{2}))\\
&\geq \PX\left(A\subset B(z,r_0), \phi_{t_0}(\cdot, B(z,r_0)) \subset  B(z,\frac{r_0}{2})\right)\\ 
&= \PX(A\subset B(z, r_0))\PX\left(\phi_{t_0}(\cdot, B(z,r_0)) \subset  B(z,\frac{r_0}{2})\right)\\
&>0.
\end{align*}
Here, we used the independence of $\mathcal{F}_0$ and $\mathcal{F}_{0,t_0}$. 
Given any $\varepsilon >0$, we may apply strongly recurrent at $z$ iteratively until we get
$$
\PX(A\subset B(z,\varepsilon))>0.
$$
The proof is complete. \qed
\end{proof}

\bigskip
With all the lemmas as above, we are in a position to give \bigskip\\
\textbf{Proof of Theorem~\ref{thm:2.1}}. 
Our strategy will be verifying  (i) $\Rightarrow$ (ii)$\Rightarrow$ (iii)  $\Rightarrow$ (iv)  $\Rightarrow$ (i)  and via applications the lemmas in this subsection.

By Lemma~\ref{lem:5.1} and Lemma~\ref{lem:5.3}, we get (i) $\Rightarrow$ (ii).
To see (ii) $\Rightarrow$ (iii), we use the fact that strongly swift transitive implies strongly recurrent, see Definition~\ref{def:loc-ST}.
Thanks to Lemma~\ref{lem:5.4}, we have (iii) $\Rightarrow$ (iv).
Due to Lemma~\ref{lem:5.2}, we obtain (iv) $\Rightarrow$ (i). 
The proof is complete. \qed

\begin{remark}
In Lemmas~\ref{lem:5.2} and~\ref{lem:5.4}, we need not assume that $E$ is locally compact and $\sigma$-compact. So Corollary~\ref{cor:2.1} is also established.
\end{remark}


\subsection{\bf Proof of Theorem~\ref{thm:2.2} }  \label{proof 2.2}


The proof follows the idea in~\cite{FGS17}.  For the reader’s convenience, we give a concise
proof here.
 Since $\phi$ is strongly mixing and weakly asymptotically stable, by~\cite[Proposition 2.20]{FGS17}, there are $\mathcal{F}_0$-measurable random variables $a_i(\omega)$, $i = 1, \ldots, N$ such that
$$
A(\omega) := \{a_i(\omega): i = 1, \ldots, N\}
$$
is a minimal weak point attractor.  
Assume $A(\omega)$ is not a singleton $\PX$-a.s. Then
\[
F(\omega) := \min_{i, j = 1, \ldots, N, \, i \neq j } d(a_i(\omega), a_j(\omega)) > 0,\quad \PX-\mathrm{a.s.}  
\]
Since $\PX$ is invariant under $\theta_t$, we have  for any $t>0$,
\begin{align}\label{eq:5.3}
 0<\EX (1\wedge F(\omega))=\EX (1\wedge F(\theta_t\omega))
\leq \EX (1\wedge d(\phi_t(\omega, a_1(\omega)), \phi_t(\omega, a_2(\omega))).
\end{align}

By weak asymptotic stability on $B(z,\varepsilon)$,  there is a sequence $t_n \to \infty$  such that for all $x, y \in B(z,\varepsilon)$ and all $\eta > 0$
\begin{equation} \label{eq:5.4}
    \liminf_{n \to \infty} \mathbb{P} \left( d \left( \varphi_{t_n}(\cdot, x), \varphi_{t_n}(\cdot, y) \right) \leq \eta \right) \geq \delta=\PX(\mathcal{M}) > 0. 
\end{equation}
Let $x, y \in E$.  Define the stopping time
\[
    \tau_\varepsilon^{x,y}(\omega) := \inf \left\{ t \geq 0 : d \left( \varphi_t(\omega, x), \varphi_t(\omega, y) \right) \leq \frac{\varepsilon}{4} \right\}.
\]
Let  $a(\omega), b(\omega) \in A(\omega)$ be two $\mathcal{F}_0$-measurable selections and let $\tau_\varepsilon(\omega) := \tau_\varepsilon^{a(\omega),b(\omega)}$, where $\tau_\varepsilon^{x,y}$ is defined as above. By~\eqref{eq:2.3}, $\tau_\varepsilon$ is finite $\mathbb{P}$-almost surely. Due to independence of $\mathcal{F}_0$ and $\mathcal{F}_{0,\infty}$, right-continuity of the trajectories implies that there is a $\iota : \Omega \to \mathbb{R}_+ \setminus \{0\}$ such that
\begin{equation*}
    d \left( \varphi_{\tau_\varepsilon(\omega) + t}(\omega, a(\omega)), \varphi_{\tau_\varepsilon(\omega) + t}(\omega, b(\omega)) \right) \leq \frac{\varepsilon}{3} 
\end{equation*}
for all $t \in [0, \iota(\omega)]$, $\mathbb{P}$-a.s. Hence, there is a $t_0 \geq 0$ such that
\[
    \mathbb{P} \left( d \left( \varphi_{t_0}(\cdot, a(\cdot)), \varphi_{t_0}(\cdot, b(\cdot)) \right) \leq \frac{\varepsilon}{3} \right) > 0.
\]
By local pointwise strong swift transitivity and using that $\varphi$ is a white noise RDS, there is a time $t_1 \geq 0$ such that
\[
    \mathbb{P} \left( \varphi_{t_0 +t_1}(\cdot, \{a(\cdot), b(\cdot)\}) \subseteq B(z,\varepsilon) \right) > 0.
\]
Using $\varphi$ is a white noise RDS again and \eqref{eq:5.4}, we have for any $\eta>0$
\begin{equation}\label{eq:5.5}
    \liminf_{n \to \infty} \mathbb{P} \left( d \left( \varphi_{t_0 + t_1 + t_n}(\cdot, a(\cdot)), \varphi_{t_0 + t_1 + t_n}(\cdot, b(\cdot)) \right) \leq \eta \right)  > 0.  
\end{equation}
Which is in contradiction to~\eqref{eq:5.3}.\qed


\subsection{\bf Proof of Theorem~\ref{thm:2.3} }  \label{proof 2.3}


The proof of (i) $\Rightarrow$ (ii) is straightforward.  We only need to prove (ii) $\Rightarrow$ (i).
Since $\phi$ is strongly mixing and weakly asymptotically stable, by~\cite[Proposition 2.20]{FGS17}, there are $\mathcal{F}_0$-measurable random variables $a_i(\omega)$, $i = 1, \ldots, N$ such that
$$
A(\omega) := \{a_i(\omega): i = 1, \ldots, N\}
$$
is a minimal weak point attractor.  

Assume $A(\omega)$ is not a singleton $\PX$-a.s. Then
\[
F(\omega) := \min_{i, j = 1, \ldots, N, \, i \neq j } d(a_i(\omega), a_j(\omega)) > 0,\quad \PX-\mathrm{a.s.}  
\]
Since $\PX$ is invariant under $\theta_t$, we have 
\begin{align}\label{eq:5.2}
 0<\EX (1\wedge F(\omega))=\EX (1\wedge F(\theta_t\omega))
\leq \EX (1\wedge d(\phi_t(\omega, a_1(\omega)), \phi_t(\omega, a_2(\omega))).
\end{align}
Due to the freezing lemma (\cite[Lemma 4.1]{Baldi}) and the fact that $\mathcal{F}_0$ is independent of $\mathcal{F}_{0,\infty}$, we have
\begin{align*}
 \EX (1\wedge d(\phi_t(\omega, a_1(\omega)), \phi_t(\omega, a_2(\omega)))
 &=\EX [\EX (1\wedge d(\phi_t(\omega, a_1(\omega)), \phi_t(\omega, a_2(\omega)))|\mathcal{F}_0]\\
 &=\EX [\EX (1\wedge d(\phi_t(\omega, x), \phi_t(\omega, y))|_{x=a_1(\omega),y=a_2(\omega)}].
\end{align*}
Noticing
$$
\lim_{t \to \infty}  d(\phi_t(\omega, x), \phi_t(\omega, y)) = 0 \quad \text{ in probability}, 
$$
and using the dominated convergence theorem, we get
$$
\lim_{t \to \infty} \EX (1\wedge d(\phi_t(\omega, a_1(\omega)), \phi_t(\omega, a_2(\omega)))=0,
$$
which is contradict to \eqref{eq:5.2}. \qed

\section*{Appendix} \appendix

\renewcommand{\theequation}{\thesection.\arabic{equation}}

\section{Random dynamical systems} \label{RDS}

Let $(E, d)$ be a complete separable metric space with Borel $\sigma$-algebra $\mathcal{E}$ and $(\Omega, \mathcal{F}, \PX, \theta)$ be an ergodic metric dynamical system, i.e., $(\Omega, \mathcal{F}, \PX)$ is a 
probability space and $\theta := (\theta_t)_{t \in \mathbb{R}}$ is a group of jointly measurable maps on $(\Omega, \mathcal{F}, \PX)$ with ergodic invariant measure $\PX$.

Let $\phi : \mathbb{R} \times \Omega \times E \rightarrow E$ be a perfect cocycle: i.e., $\phi$ is measurable, $\phi_0(\omega, x) = x$ and $\phi_{t+s}(\omega, x) = \phi_{t}(\theta_{s}\omega, \phi_s(\omega,x))$ for all $x \in E$, $t, s \geq 0$, $\omega\in \Omega$. We will assume that $\phi_s(\omega, \cdot)$ is continuous for all $s \geq 0$ and $\omega \in \Omega$. The collection $(\Omega, \mathcal{F}, \PX, \theta, \phi)$ is then called a random dynamical system (RDS), see~\cite{Arn98} for a comprehensive treatment.
 
Consider an RDS generated by an SDE driven by white noise, we assume that $\theta_t\omega(\cdot)=\omega(\cdot+t)-\omega(t)$.  We have a family $\mathcal{F} = (\mathcal{F}_{s,t})_{-\infty < s \leq t < \infty}$ of sub-$\sigma$ algebras of $\mathcal{F}$ such that $\mathcal{F}_{t,u} \subset \mathcal{F}_{s,v}$ whenever $s \leq t \leq u \leq v$, $\theta_r(\mathcal{F}_{s,t}) = \mathcal{F}_{s+r,t+r}$ for all $r, s, t$, and $\mathcal{F}_{s,t}$ and $\mathcal{F}_{u,v}$ are independent whenever $s \leq t \leq u \leq v$. For each $t \in \mathbb{R}$, let us denote the smallest $\sigma$-algebra containing all $\mathcal{F}_{s,t}$, $s \leq t$, by $\mathcal{F}_t$ and the smallest $\sigma$-algebra containing all $\mathcal{F}_{t,u}$, $t \leq u$, by $\mathcal{F}_{t,\infty}$. Note that for each $t \in \mathbb{R}$, the $\sigma$-algebras $\mathcal{F}_t$ and $\mathcal{F}_{t,\infty}$ are independent. We will further assume that $\phi_s(\cdot, x)$ is $\mathcal{F}_{0,s}$-measurable for each $s \geq 0$. The collection $(\Omega, \mathcal{F}, \mathcal{F}_t, P, \theta, \phi)$ is then called a white noise (filtered) random dynamical system.

An invariant measure for an RDS $\phi$ is a probability measure on $\Omega \times E$ with marginal $\PX$ on $\Omega$ that is invariant under $\phi$ for $t\geq 0$. For each probability measure $\mu$ on $\Omega \times E$ with marginal $\PX$ on $\Omega$, there is a unique disintegration $\mu_{\omega}$ such that $\mu$ is an invariant measure for $\phi$ if and only if $\phi_{t}(\omega)\mu_{\omega} = \mu_{\theta_{t}\omega}$ for all $t \geq 0$, almost all $\omega \in \Omega$. An invariant measure $\mu$ is said to be a Markov measure if $\omega \mapsto \mu_{\omega}$ is measurable concerning the past $\mathcal{F}_{0}$. In case of a white noise RDS $\phi$, we may define the associated Markovian semigroup by 
$$
P_{t}f(x) := \mathbb{E}f(\phi_{t}(\cdot, x)).
$$
Where $f$ is measurable and bounded.

We say that a Markovian semigroup $P$ with invariant measure $\rho$ is strongly mixing if, for each continuous, bounded function $f$ and all $x \in E$, we have
\[
P_tf(x) \rightarrow \int_E f(y) \, \dif \rho(y), \quad \text{as} \quad t \to \infty.
\]
Similarly, we say that an RDS $\varphi$ is strongly mixing if the law of $\varphi_t(\cdot, x)$ converges to $\rho$ as $t \to \infty$ for all $x \in E$.

Next, we recall the definition of a pullback attractor and a weak random attractor~\cite{Sc, Crauel97, Crauel01}.
 
\begin{definition}
Let $(\Omega, \mathcal{F}, P, \theta, \phi)$ be an RDS. A random, compact set $A$ is called a pullback attractor if:\\
 {\rm$(1)$}  $A$ is $\phi$-invariant, and\\
 {\rm$(2)$} For every compact set $B \in E$, we have
    \[
    \lim_{ t \to \infty} \sup_{x\in B} d(\phi_t(\theta_{-t}\omega, x), A(\omega)) = 0, \quad \text{almost surely}.
    \]
\end{definition}

The map $A$ is called a weak attractor if it satisfies the above properties with almost sure convergence replaced by convergence in probability in $(2)$. It is called a (weak) point attractor if it satisfies the properties above with compact sets $B$ replaced by single points in $(2)$. A (weak) point attractor is said to be minimal if it is contained in each (weak) point attractor.
 
Every pullback attractor is a weak attractor. Weak attractors are unique, that is, if an RDS has two weak attractors, they agree almost surely. 

Consider an RDS generated by an SDE driven by white non-Gaussian noise,  we need to recall the Skorokhod space. Let
$D([0,T],\mathbb{R}^d)$ be the space of c\`{a}dl\`{a}g $\R^d$-valued
functions on $[0, T ]$. 
Let $D_{J_1}([0,T],\mathbb{R}^d)$ be  $D([0,T],\mathbb{R}^d)$ equipped
with the $J_1$-metric, see~\cite{Billingsley}. Let
$D_{U}([0,T],\mathbb{R}^d)$ denote $D([0,T],\mathbb{R}^d)$ equipped
with the uniform metric~\cite{Pollard2012}.
 
\bigskip
{\bf Acknowledgement.}  
The authors acknowledge the support provided by NNSFs of China (No. 11971186).


\end{document}